\documentclass[12pt]{article}
\usepackage{amsmath,amsthm,amssymb,fullpage,enumerate,cleveref,url,titlesec}

\newtheorem{theorem}{Theorem}[section]
\newtheorem{lemma}[theorem]{Lemma}
\newtheorem{definition}[theorem]{Definition}
\newtheorem{corollary}[theorem]{Corollary}

\newcommand{\eps}{\varepsilon}
\newcommand{\F}{\mathcal{F}}

\newcommand{\C}{\mathcal{C}}

\DeclareMathAccent{\widehat}{\mathord}{largesymbols}{"62}

\title{Set System Blowups}
\author{
Ryan Alweiss \thanks{Research supported by an NSF Mathematical Sciences Postdoctoral Fellowship.} \\
Department of Mathematics\\
University of Cambridge\\
\texttt{ra699@cam.ac.uk}
}

\begin{document}
\maketitle

\begin{abstract}
	We prove that given a constant $k \ge 2$ and a large set system $\F$ of sets of size at most $w$, a typical $k$-tuple of sets $(S_1, \cdots, S_k)$ from $\F$ can be “blown up” in the following sense: for each $1 \le i \le k$, we can find a large subfamily $\F_i$ containing $S_i$ so that for $i \neq j$, if $T_i \in \F_i$ and $T_j \in \F_j$ , then $T_i \cap T_j=S_i \cap S_j$.  We also show that the answer to the multicolor version of the sunflower conjecture is the same as the answer for the original, up to an exponential factor.

\end{abstract}

\section{Introduction}
\label{s1}
\subsection{Background and Main Result}
In 2019, the author, Lovett, Wu, and Zhang improved the best known bounds for the sunflower lemma \cite{alwz}.  Central to this result was an ``encoding" argument, in which the ground set $X$ of a set system $\F$ is colored in some way. Most pairs $(S,\C)$ of a set $S \in \F$ and a coloring $\C$ of the ground set are then shown to have some property, by the construction of an explicit injection of the ``bad" pairs that do not have the property into some small collection.
In this note, we use this encoding idea to prove a rather surprising fact about set systems. A \emph{$w$-set system} is a family of sets, each of size at most $w$. Throughout this part, let $\F$ be a $w$-set system on a ground set $X$, and let $k \ge 2$ be some fixed integer. We introduce sunflowers and describe them in detail in~\Cref{sfs}. \begin{definition}
Call a $k$-tuple of sets $(S_1, \cdots, S_k)$ from $\F$ $n$-\emph{inflatable} if for each $1 \le i \le k$ there exists a subfamily $\F_i \ni S_i$ of $\F$ so that $|\F_i| \ge |\F|/n$, and so that for any $T_i \in \F_i$ and $T_j \in \F_j$ with $i \neq j$, we have $T_i \cap T_j = S_i \cap S_j$.  
\end{definition}
We show, that in a large set system almost all $k$-tuples are inflatable, where the corresponding families $\F_i$ are very large.

\begin{theorem}
\label{thm:main}
Let $k \ge 2$ be fixed, and let $\F$ be a $w$-set system on some ground set $X$.  For all choices of $n$, all but at most $\dbinom{k+w-1}{k-1}\frac{k^{w+1}2^{w(k-1)}}{n}|\F|^k$ of the $|\F|^k$ tuples $(S_1, \cdots, S_k) \in \F^k$ are $n$-inflatable.
\end{theorem}

It would be quite interesting to reproduce any of the results proved here in some other way, and perhaps this would shed some light on the sunflower conjecture. 

\subsection{Proof Outline}

We prove \Cref{thm:main} by introducing the concept of ``mimicking" sets, which allows us to capture inflatability through looking at colorings.  Given a $k$-tuple $(S_1, \cdots, S_k)$ from $\F$ and a coloring $\C$ on $X$, we say that $T$ ``mimics" $S_j$ if it satisfies certain natural coloring conditions, and its intersections with the $S_i$ are the same as those of $S_j$. If enough sets mimic each $S_j$, we will have the desired inflatability property.  We can use an encoding argument to count the number of ``bad" $k$-tuples $(S_1, \cdots, S_k)$ so that some $S_j$ does not have too many sets $T$ mimicking it.
	
\subsection{Blowing up Pairs of Sets}

For $k=2$, \Cref{thm:main} yields the following.

\begin{theorem}[\Cref{thm:main}, $k=2$]
\label{thm:2set} If $\F$ is a $w$-set system, then all but at most $\frac{4^w(2w+2)}{n}|\F|^2$ pairs $(S, T)$ of sets from $\F$ are $n$-inflatable. \end{theorem}

We make note of a particularly interesting consequence. First, recall the definition of a link. Given a set system $\F$ on $X$ and a set $U \subset X$, the \emph{link} of $\F$ at $U$ is
$\F_U = \{S \setminus U: S \in \F, U \subset S\}.$

We are now ready to state a corollary of \Cref{thm:2set}.

\begin{corollary}
\label{corollary:bipartite}
If $S, T$ are chosen randomly from $\F$, then with probability at least $1-\frac{4^w(2w+2)}{n}$, the link $\F_U$, where $U=S \cap T$, contains two cross-wise disjoint subfamilies $\F_1$, $\F_2$ each of size at least $|\F|/n$. \end{corollary}

One could for instance take $n=5^w$ so that the link of the intersection of almost any (as $w$ goes to infinity) pair of sets contains two cross-wise disjoint subfamilies of size $|\F|/5^w$.  Another consequence of \Cref{corollary:bipartite} is that in any induced subgraph of a Kneser graph $KG(m,w)$, a typical edge is contained in a large complete bipartite subgraph.  Recall that the Kneser graph $KG(m,w)$ has a vertex corresponding to each subset of $[m]$ of size $w$, and two vertices are adjacent if and only if their corresponding subsets are disjoint.

We state this explicitly as a theorem.

\begin{theorem}
\label{thm:kneser}
	Let $G=KG(m,w)$ be a Kneser graph, and let $G'$ be an induced subgraph.  Then for any $n$, and for all but at most $\frac{4^w(2w+2)}{n}|V(G')|^2$ edges $e$ of $G'$, there is some complete bipartite subgraph of $G'$ with $|V(G')|/n$ vertices on each side that contains $e$. \end{theorem}
	
In particular, if we have a $w$-set system $\F$ with more than $\eps|\F|^2$ pairs of disjoint sets, then $n=\frac{1}{\eps}4^w(2w+2)$ in \Cref{thm:kneser} yields that there are some disjoint subfamilies $\F_1$ and $\F_2$ of size at least $\frac{\eps|\F|}{4^w(2w+2)}$.  In other words, any set system with many pairs of disjoint sets has a large pair of disjoint subfamilies.

For general $k$, we have a similar result for Kneser $k$-hypergraphs.  In other words, if many $k$-tuples of a set system are pairwise disjoint, we can find some disjoint subfamilies $\F_1, \cdots, \F_k$ which are large.

\subsection{Sunflower Blowups}
\label{sfs}

We will also use \Cref{thm:main} to extend the sunflower lemma of Erd\H{o}s and Rado \cite{ErdosR1960}.  Erd\H{o}s and Rado originally called sunflowers $\Delta$-systems, but Deza and Frankl \cite{deza1981every} coined the name ``sunflower". 

\begin{definition}
	For a fixed $k \ge 2$, a $k$-petal \emph{sunflower} is a $k$-tuple of sets $(S_1, \cdots, S_k)$ so that for any $1 \le i<j \le k$, $S_i \cap S_j=S_1 \cap \cdots \cap S_k$.  The $S_i \setminus (S_1 \cap \cdots \cap S_k)$ are called the \emph{petals}, and $S_1 \cap \cdots \cap S_k$ is called the \emph{kernel}. \end{definition}
	
The case $k=2$ is trivial, because any pair of sets form a $2$-petal sunflower, so typically we assume $k \ge 3$.  We will also prove a corollary.

\begin{corollary}
\label{corollary:sf}
	Let $f_k(w)$ be such that any $w$-set system with $f_k(w)$ sets contains a $k$-petal sunflower. Then, there exists some constant $D_k$ depending on $k$ so that if $\F$ is a $w$-set system on a ground set $X$, then almost all (as $w$ goes to infinity) of its $k$-petal sunflowers $(S_1, \cdots, S_k)$ are $n$-inflatable with parameter $n=f_k(w)^kD_k^w$.
\end{corollary}

The sunflower conjecture (\cite{ErdosR1960}) states that we may take $f_k(w)=(O_k(1))^w$. By using the modifications of the original argument of the author, Lovett, Wu, and Zhang \cite{alwz} due to Frankston, Kahn, Narayanan, and Park \cite{fknp}, Rao \cite{rao2019coding} proved that we may take $f_k(w)=(O_k(\log w))^w$. Thus, if the sunflower conjecture is true, there is some constant $C_k$ depending on $k$ so that if $\F$ is a $w$-set system, then for a typical $k$-petal sunflower $(S_1, \cdots, S_k)$, there exists subfamilies $\F_1, \cdots, \F_k$ of $\F$ of size at least $|\F|/C_k^w$ so that any $S_1 \in \F_1, \cdots, S_k \in \F_k$ form a sunflower with kernel $S_1 \cap \cdots \cap S_k$. Unconditionally, there is some constant $C_k$ depending on $k$ such that a typical $k$-petal sunflower can be blown up to yield families $\F_1, \cdots, \F_k$ of $\F$ of size at least $|\F|/(C_k(\log w)^k)^w$ so that any $S_1 \in \F_1, \cdots, S_k \in \F_k$ form a $k$-petal sunflower.

\section{Main Technical Proofs}
\label{s2}

We assume $w \ge 2$; $w=1$ is a trivial case. We can assume all sets are of size exactly $w$ by adding dummy elements to sets that are too small, since this does not affect inflatability of any $k$-tuple. We may also assume the ground set $X$ has size divisible by $k$, since we may add dummy elements to $X$ to make this the case. Thus, we may randomly partition $X$ into $k$ equal-sized parts $X_i$ for $1 \le i \le k$.  We say that the elements of $X_i$ have \emph{color} $i$ and we refer to the partition of $X$ into the $X_i$ as a \emph{coloring} $\C$ of the ground set $X$.  We call a coloring \emph{balanced} if there are the same number of elements of each of the $k$ colors.
	
\begin{definition}
Fix a $w$-set system $\F$ on a ground set $X$, a balanced coloring $\C$ of $X$, and a $k$-tuple $(S_1, \cdots, S_k)$ in $\F^k$. We say that $T \in \F$ \emph{mimics $S_j$ with respect to $(S_1, \cdots, S_k)$} if for all $i \neq j$, $S_i \cap T=S_i \cap S_j$, and $T \setminus S_j \subset X_j$.
\end{definition}

The last condition, that all elements of $T \setminus S_j$ are colored with the color $j$, at first glance appears to be rather restrictive.  It turns out that this is surprisingly not the case, in the sense that for a random $k$-tuple, many sets will mimic each $S_j$.

\begin{definition}
For a fixed $1 \le j \le k$, say that $(S_1, \cdots, S_k)$ is \emph{$(n,j)$-bad} (under a balanced coloring $\C$) if there are fewer than $|\F|/n$ sets $T \in \F$ such that $T$ mimics $S_j$ with respect to $(S_1, \cdots, S_k)$. Say $(S_1, \cdots, S_k)$ is \emph{$n$-bad} under $\C$ if it is $(n,j)$-bad under $\C$ for some $1 \le j \le k$. A $k$-tuple is \emph{$n$-good} under $\C$ if it is not $n$-bad under $\C$. \end{definition}

The following lemma shows the relationship between the good and inflatable properties:
\begin{lemma} If $(S_1, \cdots, S_k)$ is $n$-good under some coloring $\C$, then it is $n$-inflatable.
\label{lemma:color2inflate}	
\end{lemma} 

Therefore, it suffices to prove the following lemma:	
\begin{lemma}
\label{lemma:colorup}
	Let $\F$ be a set system on a ground set $X$. For any $1 \le j \le k$, if $\C$ is a uniform balanced coloring of $X$ and $(S_1, \cdots, S_k)$ is sampled uniformly from $\F^k$, then $(S_1, \cdots, S_k)$ is $n$-bad with probability at most $\dbinom{k+w-1}{k-1}\frac{k^{w+1}2^{w(k-1)}}{n}$. \end{lemma}
	
The proof of \Cref{lemma:colorup} in \Cref{pfcolorup} follows from an encoding argument similar in spirit to that of \cite{alwz}.  For each fixed $1 \le j \le n$, this allows to explicitly bound the number of pairs of a balanced coloring $\C$ and a $k$-tuple $(S_1, \cdots, S_k)$ so that $(S_1, \cdots, S_k)$ is $(n,j)$-bad under $\C$.

\subsection{Proof of \Cref{lemma:color2inflate}}

	Take $\F_j$ to be the family of $T_j$ which mimic $S_j$ under the coloring $\C$.  Note $|\F_j| \ge |\F|/n$ for each $1 \le j \le k$ by assumption. It suffices to show that if $T_i \in \F_i$ and $T_j \in \F_j$ for $i \neq j$, then we have $T_i \cap T_j=S_i \cap S_j$. Clearly $T_i \supset T_i \cap S_j=S_i \cap S_j$, and similarly $T_j \supset T_j \cap S_i=S_i \cap S_j$.  Hence, $T_i \cap T_j \supset S_i \cap S_j$.  Now, we prove $T_i \cap T_j \subset S_i$.  If $x \in T_i \cap T_j$ is not in $S_i$, then since $T_i \setminus S_i \subset X_i$, we have $x \in X_i$.  Because $X_i$ and $X_j$ are disjoint, $x \notin X_j \supset T_j \setminus S_j$, and hence $x \notin T_j \setminus S_j$.  Since $x \in T_j$, we have $x \in S_j$.  But then $x \in T_i \cap S_j=S_i \cap S_j$, a contradiction.  Hence, $T_i \cap T_j \subset S_i$, and similarly $T_i \cap T_j \subset S_j$, so $T_i \cap T_j=S_i \cap S_j$.

\subsection{Proof of \Cref{lemma:colorup}}
\label{pfcolorup}
For a fixed $1 \le j \le k$, we will bound the number of pairs of balanced colorings $\C$ and $k$-tuples $(S_1, \cdots, S_k) \in \F^k$ so that $(S_1, \cdots, S_k)$ is $(n,j)$-bad under $\C$. We will do this by an ``encoding" argument; we can recover this pair from the following information.
\begin{enumerate}
\item 
\label{color} 
The first piece of information will be $\C'$, which is the coloring obtained by taking $\C$ and then recoloring the elements of $S_j$ with the color $j$. We claim that the number of possibilities for this is at most $\dbinom{|X|}{|X|/k, \cdots, |X|/k}\dbinom{k+w-1}{k-1}$.  This is because there are at most $\dbinom{k+w-1}{k-1}$ possibilities for the number of times that each of the $k$ colors appears as one of the $w$ elements of $S_j$, and by the log-convexity of the factorial function, at most $\dbinom{|X|}{|X|/k, \cdots, |X|/k}$ possibilities for $\C'$ once the number of elements of each color is fixed.
\item 
\label{othersets}
The second piece of information will be $S_i$ for each $i \neq j$.  There are $|\F|^{k-1}$ possibilities for this.
\item 
\label{cap}
The third piece of information will be $S_i \cap S_j$ for each $i \neq j$.  There are at most $2^{w(k-1)}$ possibilities for this, as for each of the $k-1$ sets $S_i$ with $i \neq j$ we specify one of $2^w$ possible subsets for $S_i \cap S_j$.
\item 
\label{index1}
The fourth piece of information specifies which of the mutually mimicking sets $T_j$ is $S_j$.  The sets $T_j$ so that $T_j \cap S_i=S_j \cap S_i$ for each $i \neq j$ and so that all elements $T_j$ have color $j$ in $\C'$ must have $T_j \setminus S_j \subset X_j$.  Hence, they must mimic $S_j$ with respect to $(S_1, \cdots, S_k)$, and so there are at most $|\F|/n$ possibilities for them by assumption.  We can thus identify $S_j$ by a positive integer at most $|\F|/n$.
\item
\label{index2}
The last piece of information specifies the coloring of the elements of $S_j$ in the original coloring.  There are at most $k^w$ possibilities, since $|S_j| \le w$ and each element of $S_j$ can be colored with one of $k$ colors.\footnote{Technically we could obtain the bound $\dbinom{w}{w/k, \cdots, w/k}$ here, omitting non-crucial floors and ceilings, but $k$ will generally be small compared to $w$, so we will not save much by doing that.  We use the estimate $k^w$ for ease of notation and readability.}
\end{enumerate}

Thus, of the $\dbinom{|X|}{|X|/k, \cdots, |X|/k}|\F|^k$ possible pairs of $\C$ and $(S_1, \cdots, S_k)$, the number so that $(S_1, \cdots, S_k)$ is $(n,j)$-bad under $\C$ is at most $$k^w\dbinom{|X|}{|X|/k, \cdots, |X|/k}|\F|^{k-1}\dbinom{k+w-1}{w-1}2^{w(k-1)}|\F|/n$$ $$= k^w\left(\dbinom{|X|}{|X|/k, \cdots, |X|/k}|\F|^k \right)\left(2^{w(k-1)}\dbinom{k+w-1}{w-1}\frac{1}{n}\right).$$

By a simple union bound over $k$ choices of $1 \le j \le k$, the number of possible pairs of $\C$ and $(S_1, \cdots, S_k)$ so that $(S_1, \cdots, S_k)$ is $n$-bad under $\C$ is at most $$\left(\dbinom{|X|}{|X|/k, \cdots, |X|/k}|\F|^k \right)\left(\dbinom{k+w-1}{w-1}k^{w+1}2^{w(k-1)}\frac{1}{n}\right).$$  This completes the proof of \Cref{lemma:colorup}.

\subsection{Proof of \Cref{thm:main}}

Using \Cref{lemma:colorup}, there exists a specific coloring $\C$ under which at most $$\dbinom{k+w-1}{k-1}k^{w+1}2^{w(k-1)}|\F|^k/n$$ of the $|\F|^k$ $k$-tuples $(S_1, \cdots, S_k)$ are $n$-bad.  In combination with \Cref{lemma:color2inflate}, this immediately implies \Cref{thm:main}.

\subsection{Proof of \Cref{corollary:bipartite}} If $(S, T)$ is $n$-inflatable, then there exists subfamilies $\F(S) \ni S$ and $\F(T) \ni T$ of $\F$ of size at least $|\F|/n$ such that $S' \cap T'=S \cap T=U$ for all $S' \in \F(S), T' \in \F(T)$.  Then for all $S' \in \F(S), T' \in \F(T)$, we have $U \subset S', U \subset T'$, and $(S' \setminus U) \cap (T' \setminus U)=(S' \cap T') \setminus U=U \setminus U=\emptyset$.  Thus if we take $\F_1=\{S' \setminus U: S' \in \F(S)\}$ and $\F_2=\{T' \setminus U: T' \in \F(T)\}$, we are done.

\subsection{Proof of \Cref{thm:kneser}}

Let $\F=V(G')$.  Then by \Cref{thm:2set}, all but at most $\frac{4^w(2w+2)}{n}|V(G')|^2$ pairs of sets of $\F$ are inflatable.  If $(S,T)$ is $n$-inflatable and $S,T$ are disjoint, then this means exactly that the edge $e$ between $S$ and $T$ in $G'$ is contained in a $K_{|V(G')|/n, |V(G')|/n}$ in $G'$.

\subsection{Proof of \Cref{corollary:sf}}
By \Cref{thm:main}, all but at most 

$$\dbinom{k+w-1}{k-1}\frac{k^{w+1}2^{w(k-1)}}{n}|\F|^k=k^{w+1}2^{w(k-1)}\frac{|\F|^k}{n}\prod_{i=1}^{w} \frac{k+i-1}{i} $$ $$\le k^w\frac{k^{w+1}2^{w(k-1)}}{n}|\F|^k \le \frac{k^wk^{2w}2^{wk}}{n}|\F|^k=\frac{(k^32^k)^w}{n}|\F|^k$$ 

of the $|\F|^k$-tuples $(S_1, \cdots, S_k) \in \F^k$ are $n$-inflatable.

In particular, there exists a constant $C_k=k^32^k$ so that all but at most $\frac{C_k^w}{n}|\F|^k$ of the $k$-tuples of sets from $\F$ are $n$-inflatable. Now, if any family of $f_k(w)$ sets contains a $k$-petal sunflower, then by a simple averaging argument, a $w$-set system $\F$ so that $|\F| \ge f_k(w)$ will have at least $\left(\frac{|\F|}{f(w)}\right)^k$ $k$-petal sunflowers. If $D_k=100C_k$, then if $n=f(w)^kD_k^w$,  almost all (as $w$ goes to infinity) $k$-petal sunflowers are inflatable.

\section{Multicolored Sunflowers}

We conclude this paper with a remark about multicolor sunflowers, which is of a somewhat similar flavor to the blown-up sunflowers that we consider in this paper.  Consider the \emph{multicolor} variant of the sunflower problem.  Say that for $1 \le i \le n$, there exists a triple of $(A_i,B_i,C_i)$ of sets of size $w$, so that $(A_i,B_j,C_k)$ form a sunflower if and only if $i=j=k$.  In terms of $w$, how big is the largest such $n$ for which this can exist?  If we have a sunflower-free family $\F$  with $n=|\F|$ sets $\{F_1, \cdots, F_n\}$, setting $A_i=B_i=C_i=F_i$ works, so such a family can be at least as large as the sunflower bound.

We mention that for the multicolor version of Roth's theorem, i.e. finding the largest number of triples $(a_i, b_i, c_i)$ in $[n]$ so that $a_i-2b_j+c_k=0$ if and only if $i=j=k$, the best bounds we are aware of are only from the triangle removal lemma.  Over $(\mathbb{Z}/3)^n$, by contrast, it is known \cite{kss} that the polynomial method gives the tight answer to the multicolor version of the capset problem.

In the case of the sunflowers, it turns out that the two problems are equivalent, up to an exponential factor in $w$.

\begin{theorem}
	Let $\F$ be a system of triples $(A_i, B_i, C_i)$ of sets of size at most $w$ so that $(A_i, B_j, C_k)$ form a sunflower if and only if $i=j=k$.  If any family of $f_3(w)$ sets of size at most $w$ contains a normal sunflower, then $|\F| \le (54e^3+o_w(1))^wf_3(w)$.
\end{theorem}

Essentially the same proof works for sunflowers with an arbitrary number of $k$ petals and a bound of $(2k^ke^k+o_w(1))^wf_k(w)$, but with slightly more cumbersome notation, so we just present the slightly more readable and less general proof, with $k=3$.  The constant $(54e^3)^w$ (or $(2(ek)^k)^w$) is surely not tight.  It would be interesting to find the best constant here, and also to find the best bounds in Theorem~\ref{thm:main}. We hope that all of this will shed some more light on the sunflower conjecture.

\begin{proof}

Say there is such a family $(A_i, B_i,C_i)$ of size $n$.  Color the ground set $X$ randomly, giving each element of $X$ a random uniform color from $[w]$.  Keep only the $(A_i,B_i,C_i)$ so that all of $A_i, B_i, C_i$ have exactly one element of each color.  In expectation and thus for some coloring, we keep at least a $(e^{-3}+o_w(1))^{-w}$ proportion of all such $(A_i, B_i, C_i)$ .  Pass to such a subset.  Now, let $S_i=A_i \cap B_i \cap C_i$.  There are at most $2^w$ choices for the colors that appear in $S_i$.  Losing at most a factor of $2^w$, we can pass to a sub-family formed by the subset of $i$ on which this set of colors is fixed.  This also guarantees that $|A_i \cap B_i \cap C_i|$ is independent of $i$.

Thus we have a subfamily where each of $A_i, B_i, C_i$ has one element of each color in $[w]$, and $S_i=A_i \cap B_i \cap C_i$ all have the same set of colors.  We call these the ``inner" colors of $[w]$ and the remaining colors the ``outer" colors of $[w]$.

Now, randomly color the ground set $X$ red, blue, and green (so all of the elements are colored by $[w] \times \{red, blue, green\}$). For each $(A_i,B_i,C_i)$ there is a probability at least $27^{-w}$ that $A_i\setminus S_i$ has all elements red, $B_i\setminus S_i$ has all elements blue, and $C_i\setminus S_i$ has all elements green.  So there is some coloring where we can pass to a sub-family where everything satisfies this property, losing a factor of at most $27^w$.

Finally, once we do this, we can use the original sunflower theorem.  In this new family, if $S_i, S_j, S_k$ form a sunflower, then $A_i, B_j, C_k$ are readily seen to also form a sunflower.  An element of $A_i \cap B_j$ that is colored by one of the inner colors will also be in $S_i$ and $S_j$ and thus also will be in $S_k$.  An element of $A_i \cap B_j$ that is one of the outer colors must be in both $A_i \setminus S_i$ and $B_j \setminus S_j$, but then it must be both red and blue, a contradiction.\end{proof}


%
%
%

\section{Acknowledgments}

We thank Noga Alon, Zach Chase, Peter Frankl, Noah Kravitz, Andrey Kupavskii, Shachar Lovett, Mihir Singhal, Kewen Wu, Jiapeng Zhang, and three anonymous reviewers for helpful comments.

\end{document}